\documentclass{article}
\usepackage{graphicx}
\usepackage{subfigure}
\usepackage{amsmath}
\usepackage{caption}
\usepackage{graphicx, subfig}
\usepackage{graphicx}
\usepackage{epstopdf}
\usepackage{amsmath}
\usepackage{amssymb}
\usepackage{algorithm, algorithmic}
\usepackage{graphicx, subfig}
\usepackage{multirow, multicol}
\usepackage{abstract}
\usepackage{bm}
\usepackage{subfigure}
\usepackage{float}
\usepackage{adjustbox}
\usepackage{amsthm,amsmath,amssymb}
\usepackage{mathrsfs}
\usepackage{cite}
\usepackage{caption}
\newtheorem{theorem}{Theorem}
\newtheorem{lemma}{Lemma}

\newtheorem{example}{Example}
\newtheorem{proposition}{Proposition}
\numberwithin{equation}{section}
\numberwithin{example}{section}

\begin{document}

\date{}
\title{Preconditioned TBiCOR and TCORS Algorithms for Solving the Sylvester Tensor Equation}
\author{Guang-Xin Huang$^{1}$\thanks{Emails: huangx@cdut.edu.cn (G.X. Huang), qixinggenius@163.com (Q.X. Chen), fyin@suse.edu.cn (F. Yin)}, Qi-Xing Chen$^{1}$, Feng Yin$^{2}$ \mbox{}\\
   \emph{\small 1. College of Mathematics and Physics, Chengdu University of Technology, P.R.China}\\
   \emph{\small 2. College of Mathematics and statistics, Sichuan University of Science and Engineering, P.R.China}\\}
\maketitle
\renewcommand{\abstractname}{}
\begin{onecolabstract}

\noindent{$\bm{Abstract.}$
 In this paper, the preconditioned TBiCOR and TCORS methods are presented for solving the Sylvester tensor equation. A tensor Lanczos $\mathcal{L}$-Biorthogonalization algorithm (TLB) is derived for solving the Sylvester tensor equation. Two improved TLB methods are presented. One is the biconjugate $\mathcal{L}$-orthogonal residual algorithm in tensor form (TBiCOR), which implements the $LU$ decomposition for the triangular coefficient matrix derived by the TLB method. The other is the conjugate $\mathcal{L}$-orthogonal residual squared algorithm in tensor form (TCORS), which introduces a square operator to the residual of the TBiCOR algorithm. A preconditioner based on the nearest Kronecker product is used to accelerate the TBiCOR and TCORS algorithms, and we obtain the preconditioned TBiCOR algorithm (PTBiCOR) and preconditioned TCORS algorithm (PTCORS). The proposed algorithms are proved to be convergent within finite steps of iteration without roundoff errors. Several examples illustrate that the preconditioned TBiCOR and TCORS algorithms present excellent convergence.}

\noindent{$\bm{Key words.}$ TLB, TBiCOR, TCORS, Sylvester tensor equation, Preconditioner}

\end{onecolabstract}

\section{Introduction}\label{sec1}
This paper is concerned of the computation of the Sylvester tensor equation of the form
\begin{equation}\label{eq1.1}
\mathcal{X}\times_{1}\textbf{A}_{1}+\mathcal{X}\times_{2}\textbf{A}_{2}+...+\mathcal{X}\times_{N}\textbf{A}_{N}=\mathcal{D},
\end{equation}
where matrices $\textbf{A}_{n}\in\mathbb{R}^{I_{n}\times I_{n}}(n=1,2,...,N)$ and the tensor $\mathcal{D}\in\mathbb{R}^{I_{1}\times I_{2}\times...\times I_{N}}$ are given, and the tensor $\mathcal{X}\in\mathbb{R}^{I_{1}\times I_{2}\times...\times I_{N}}$ is unknown.

The Sylvester tensor equation \eqref{eq1.1} plays vital roles in many fields such as image processing \cite{DL96}, blind source separation \cite{NCC15} and the situation when we describes a chain of spin particles \cite{MMT17}.
If $N=2$, then \eqref{eq1.1} can be reduced to the Sylvester matrix equation
\begin{equation}\label{eq1.2}
\textbf{A}\textbf{X}+\textbf{X}\textbf{B}^{T}=\textbf{D},
\end{equation}
which has many applications in system and control theory \cite{DC05,2DC05,HSF79}.
When $N=3$, \eqref{eq1.1} becomes
\begin{equation}\label{eq1.3}
\mathcal{X}\times_{1}\textbf{A}_{1}+\mathcal{X}\times_{2}\textbf{A}_{2}+\mathcal{X}\times_{3}\textbf{A}_{3}=\mathcal{D},
\end{equation}
which often arises from the finite element \cite{L04}, finite difference \cite{ZHK03} and spectral methods \cite{LSZ09}.

Many approaches are constructed to solve the Sylvester tensor equation \eqref{eq1.1} in recent years. Chen and Lu \cite{CL12} proposed the GMRES method based on a tensor format for solving \eqref{eq1.1} and presented the gradient based iterative algorithms \cite{CL13} for solving \eqref{eq1.3}. Beik et al. \cite{MAK19} also presented some global iterative schemes based on Hessenberg process to solve the Sylvester tensor equation \eqref{eq1.3}.
Beik et al. \cite{AML20} solved the Sylvester tensor equation \eqref{eq1.1} with severely ill-conditioned coefficient matrices and considered its application in color image restoration. Heyouni et al. \cite{MFA20} proposed a general framework by using tensor Krylov projection techniques to solve high order the Sylvester tensor equation \eqref{eq1.1}. Beik et al. \cite{AMA16} proposed the Arnoldi process and full orthogonalization method in tensor form, and the conjugate gradient and nested conjugate gradient algorithms in tensor form to solve the Sylvester tensor equation \eqref{eq1.1}. When $\mathcal{D}$ in \eqref{eq1.1} is a tensor with low rank, Bentbib et al. \cite{ASM20} proposed Arnoldi-based block and global methods. Kressner and Tobler \cite{DC11} developed Krylov subspace methods based on extended Arnoldi process for solving the system of equation \eqref{eq1.1}. The perturbation bounds and backward error are presented in \cite{SWL13} for solving \eqref{eq1.1}. For more methods on other linear systems in tensor form we refer to \cite{JL13,HM20,DC10,LM20}. Using the nearest Kronecker product (NKP) in \cite{VP93}, Chen and Lu \cite{CL12} presented an efficient preconditioner for solving Eq.\eqref{eq1.1} based on GMRES in tensor form. Very recently Zhang and Wang in \cite{ZW21} gave a preconditioned BiCG (PBiCG) and a preconditioned BiCR (PBiCR) based on the nearest Kronecker product (NKP) in \cite{VP93}.

Inspired by the Lanczos biorthogonalization (LB) algorithm in \cite{Y03}, BICOR and CORS methods in \cite{CJH11} for non-symmetric linear equation, in this paper, we present two improved Lanczos $\mathcal{L}$-orthogonal algorithms in tensor form for solving the Sylvester tensor equation \eqref{eq1.1}. We further present preconditioned TBiCOR (PTLB) and preconditioned TCORS (TCORS) algorithms by using the NKP preconditioner in \cite{CL12} for solving Eq \eqref{eq1.1}. The preconditioned LB in tensor form (PTLB) is also considered.

The rest of this paper is organized as follows. Section \ref{sec2} reviews some related symbols, concepts and lemmas that will be used in the contexture.
Section \ref{sec3} presents a tensor Lanczos $\mathcal{L}$-biorthogonalization algorithm (TLB) and two improved TLB methods are shown in section \ref{sec4}. The tensor biconjugate $\mathcal{L}$-orthogonal residual(TBiCOR) and tensor conjugate $\mathcal{L}$-orthogonal residual squared(TCORS) algorithms for solving the tensor equation \eqref{eq1.1} are presented in subsections \ref{sec4.1} and \ref{sec4.2}, respectively. The convergence of the TBiCOR and TCORS methods are proved. Section \ref{sec5} presents the preconditioned TLB, TBiCOR and TCORS algorithms and the convergence of the preconditioned TBiCOR and TCORS algorithms. Section \ref{sec6} presents several examples and some conclusions are drawn in section \ref{sec7}.

\section{Preliminaries}\label{sec2}
The notations and definitions as follows are needed. For a positive integer $N$, an $N$-way or $N$th-order tensor $\mathcal{X}=(x_{i_{1}}..._{i_{N}})$ is a multidimensional array with $I_{1}I_{2}...I_{N}$ entries, where $1\leq i_{j}\leq I_{j},j=1,...,N$. $\mathbb{R}^{I_{1}\times...\times I_{N}}$ denotes the set of the $N$th-order $I_{1}\times...\times I_{N}$ dimension tensors over the real field $\mathbb{R}$, while $\mathbb{C}^{I_{1}\times...\times I_{N}}$ defines the set of the $N$th-order $I_{1}\times...\times I_{N}$ dimension tensors over the complex field $\mathbb{C}$.

Let $\mathcal{X}\times_{n}\textbf{A}$ define the $n$-mode (matrix) product of a tensor $\mathcal{X}\in\mathbb{R}^{I_{1}\times I_{2}\times...\times I_{N}}$ with a matrix $\textbf{A}\in\mathbb{R}^{J\times I_{n}}$, i.e.,
\begin{displaymath}
(\mathcal{X}\times_{n}\textbf{A})_{i_{1}...i_{n-1}ji_{n+1}...i_{N}}=\sum_{i_{n}=1}^{I_{n}}x_{i_{1}i_{2}...i_{N}}a_{ji_{n}}.
\end{displaymath}
The $n$-mode (vector) product of a tensor $\mathcal{X}\in\mathbb{R}^{I_{1}\times I_{2}\times...\times I_{N}}$ and a vector $\textbf{v}\in\mathbb{R}^{I_{n}}$ is denoted by $\mathcal{X}\overline{\times}_{n}\textbf{v}$, i.e.,
\begin{displaymath}
(\mathcal{X}\overline{\times}_{n}\textbf{v})_{i_{1}...i_{n-1}i_{n+1}...i_{N}}=\sum_{i_{n}=1}^{I_{n}}x_{i_{1}i_{2}...i_{N}}\textbf{v}_{i_{n}}.
\end{displaymath}
The inner product of $\mathcal{X}, \mathcal{Y}\in\mathbb{R}^{I_{1}\times I_{2}\times...\times I_{N}}$ is defined by
\begin{displaymath}
\langle\mathcal{X},\mathcal{Y}\rangle=\sum_{i_{1}=1}^{I_{1}}\sum_{i_{2}=1}^{I_{2}}...\sum_{i_{N}=1}^{I_{N}}x_{i_{1}i_{2}...i_{N}}y_{i_{1}i_{2}...i_{N}},
\end{displaymath}
and the norm of $\mathcal{X}$ is denoted by
\begin{displaymath}
\|\mathcal{X}\|=\sqrt{\langle\mathcal{X},\mathcal{X}\rangle}.
\end{displaymath}
Furthermore, it derives from \cite{CL12} that
\begin{equation}\label{eq2.1}
\langle\mathcal{X},\mathcal{Y}\times_{n}\textbf{A}\rangle=\langle\mathcal{X}\times_{n}\textbf{A}^{T},\mathcal{Y}\rangle.
\end{equation}
We refer more notions and definitions in \cite{TB09}.

The following results from \cite{AMA16} will be used later.
\begin{lemma}\label{lem2.1}
Suppose $\textbf{A}\in\mathbb{R}^{J_{n}\times I_{n}}$, $\textbf{y}\in\mathbb{R}^{J_{n}}$ and $\mathcal{X}\in\mathbb{R}^{I_{1}\times I_{2}\times...\times I_{N}}$, we have
\begin{equation}
\mathcal{X}\times_{n}\textbf{A}\overline{\times}_{n}\textbf{e}=\mathcal{X}\overline{\times}_{n}(\textbf{A}^{T}\textbf{y}).
\end{equation}
\end{lemma}
\begin{lemma}\label{lem2.2}
If $\mathcal{X}\in\mathbb{R}^{I_{1}\times I_{2}\times...\times I_{N}}$, then
\begin{equation}
\mathcal{X}\overline{\times}_{N}\textbf{e}_{j}=\mathcal{X}_{j}, j=1,2,...,I_{N},
\end{equation}
where $\textbf{e}_{j}$ is the $j$-th column of the $I_{N}$-order identity matrix $\textbf{E}_{I_{N}}$, and $\mathcal{X}_{j}$ denotes the $j$-th frontal slice of $\mathcal{X}$.
\end{lemma}

For $\mathcal{X}\in\mathbb{R}^{I_{1}\times I_{2}\times...\times I_{N-1}\times I_{N}}$ and $\mathcal{Y}\in\mathbb{R}^{I_{1}\times I_{2}\times...\times I_{N-1}\times I_{\widehat{N}}}$, let $\mathcal{X}\boxtimes^{(N)}\mathcal{Y}\in\mathbb{R}^{I_{N}\times I_{\widehat{N}}}$ define the $\boxtimes^{(N)}$-product of $\mathcal{X}$ and $\mathcal{Y}$, i.e.,
\begin{displaymath}
[\mathcal{X}\boxtimes^{(N)}\mathcal{Y}]_{i,j}=trace(\mathcal{X}_{::...:i}\boxtimes^{(N-1)}\mathcal{Y}_{::...:j}), N=2,3,....
\end{displaymath}
In particular, $\mathcal{X}\boxtimes^{1}\mathcal{Y}=\mathcal{X}^{T}\mathcal{Y}$ for $\mathcal{X}\in\mathbb{R}^{I_{1}}$ and $\mathcal{Y}\in\mathbb{R}^{I_{1}}$.
For any $\mathcal{X}$, $\mathcal{Y}\in\mathbb{R}^{I_{1}\times I_{2}\times...\times I_{N}}$, a straightforward computation results in
\begin{equation}\label{eq0.24}
\langle\mathcal{X}, \mathcal{Y}\rangle=trace(\mathcal{X}\boxtimes^{(N)}\mathcal{Y}), N=1,2,...,
\end{equation}
and (\ref{eq0.24}) can be represented as
\begin{equation}\label{eq2.5}
\mathcal{X}\boxtimes^{({N+1})}\mathcal{Y}=trace(\mathcal{X}\boxtimes^{(N)}\mathcal{Y}).
\end{equation}
Therefore we have
\begin{equation}\label{eq2.6}
\|\mathcal{X}\|^{2}=\langle\mathcal{X}, \mathcal{X}\rangle=\mathcal{X}\boxtimes^{(N+1)}\mathcal{X}.
\end{equation}

We also need the following results.
\begin{lemma}\label{lem2.3}(\cite{AMA16})
Let $\mathcal{Y}\in\mathbb{R}^{I_{1}\times I_{2}\times...\times I_{N}\times m}$ be an ($N+1$)-order tensor with column tensors $\mathcal{Y}_{1}, \mathcal{Y}_{2},..., \mathcal{Y}_{m}\in\mathbb{R}^{I_{1}\times I_{2}\times...\times I_{N}}$ and vector $\textbf{z}\in\mathbb{R}^{m}$. For any ($N+1$)-order tensor $\mathcal{X}$ with $N$-order column tensors $\mathcal{X}_{1}, \mathcal{X}_{2},..., \mathcal{X}_{m}$, it holds that
\begin{equation}
\mathcal{X}\boxtimes^{(N+1)}(\mathcal{Y}\overline{\times}_{(N+1)}\textbf{z})=(\mathcal{X}\boxtimes^{(N+1)}\mathcal{Y})\textbf{z}.
\end{equation}
\end{lemma}
\begin{lemma}\label{lem2.4}(\cite{MFA20})
Let $\mathcal{X}\in\mathbb{R}^{I_{1}\times I_{2}\times...\times I_{N-1}\times I_{N}}$ and $\mathcal{Y}\in\mathbb{R}^{I_{1}\times I_{2}\times...\times I_{N-1}\times I_{\widehat{N}}}$ be $N$-order tensors with column tensors $\mathcal{X}_{i}(i=1,...,I_{N})$ and $\mathcal{Y}_{j}(j=1,...,I_{\widehat{N}})$. For $\textbf{A}\in\mathbb{R}^{I_{N}\times I_{N}}$ and $\textbf{B}\in\mathbb{R}^{I_{\widehat{N}}\times I_{\widehat{N}}}$, we have
\begin{equation}\label{leq2.5}
(\mathcal{X}\times_{N}\textbf{A}^{T})\boxtimes^{(N)}(\mathcal{Y}\times_{N}\textbf{B}^{T})=\textbf{A}^{T}(\mathcal{X}\boxtimes^{(N)}\mathcal{Y})\textbf{B}.
\end{equation}
\end{lemma}

\section{A Tensor Lanczos $\mathcal{L}$-Biorthogonalization Algorithm}\label{sec3}
Define the linear operator of the form
\begin{equation}\label{loper1}
\begin{split}
\mathcal{L}:\mathbb{R}^{I_{1}\times I_{2}\times...\times I_{N}}\rightarrow\mathbb{R}^{I_{1}\times I_{2}\times...\times I_{N}},\\
\mathcal{X}\mapsto \mathcal{L}(\mathcal{X}):=\mathcal{X}\times_{1}\textbf{A}_{1}+\mathcal{X}\times_{2}\textbf{A}_{2}+...+\mathcal{X}\times_{N}\textbf{A}_{N},
\end{split}
\end{equation}
then the Sylvester tensor equation \eqref{eq1.1} can be represented as
\begin{equation}
\mathcal{L}(\mathcal{X})=\mathcal{D}.
\end{equation}
Let $\mathcal{L}^{T}$ define the dual linear operator of $\mathcal{L}$, i.e.,
\begin{equation}
\begin{split}
\mathcal{L}^{T}:\mathbb{R}^{I_{1}\times I_{2}\times...\times I_{N}}\rightarrow\mathbb{R}^{I_{1}\times I_{2}\times...\times I_{N}},\\
\mathcal{X}\mapsto \mathcal{L}^{T}(\mathcal{X}):=\mathcal{X}\times_{1}\textbf{A}_{1}^{T}+\mathcal{X}\times_{2}\textbf{A}_{2}^{T}+...+\mathcal{X}\times_{N}\textbf{A}_{N}^{T},
\end{split}
\end{equation}
then it holds that $\langle \mathcal{L}(\mathcal{X}), \mathcal{Y}\rangle=\langle \mathcal{X}, \mathcal{L}^{T}(\mathcal{Y})\rangle$
for any $\mathcal{X},\mathcal{Y}\in\mathbb{R}^{I_{1}\times I_{2}\times...\times I_{N}}$.
Define the Krylov subspaces in tensor form as follows:
\begin{equation}
\mathcal{K}_{m}(\mathcal{L},\mathcal{V}_{1})=span\{\mathcal{V}_{1},\mathcal{L}(\mathcal{V}_{1}),...,\mathcal{L}^{m-1}(\mathcal{V}_{1})\},
\end{equation}
where $\mathcal{L}^{i}(\mathcal{V}_{1})=\mathcal{L}(\mathcal{L}^{i-1}(\mathcal{V}_{1}))$, $\mathcal{L}^{0}(\mathcal{V}_{1})=\mathcal{V}_{1}$,
then we have
\begin{equation}
\mathcal{K}_{m}(\mathcal{L}^{T},\mathcal{W}_{1})=span\{\mathcal{W}_{1},\mathcal{L}^{T}(\mathcal{W}_{1}),...,(\mathcal{L}^{T})^{m-1}(\mathcal{W}_{1})\}.
\end{equation}

Algorithm \ref{alg0.1} lists the Lanczos $\mathcal{L}$-Biorthogonalization procedure in tensor form that will be used to produce two series of biorthogonalization tensors.
\begin{algorithm}
\caption{A Lanczos $\mathcal{L}$-biorthogonalization procedure in tensor form.}
\label{alg0.1}
\begin{algorithmic}
\STATE{Initial: Let $\mathcal{V}_{0}=\mathcal{W}_{0}=\mathcal{O}\in\mathbb{R}^{I_{1}\times I_{2}\times...\times I_{N}}$. Select $\mathcal{V}_{1}$ and $\mathcal{W}_{1}$ subject to $\langle\mathcal{W}_{1}, \mathcal{L}(\mathcal{V}_{1}) \rangle=1$. Set $\delta_{1}=\beta_{1}=0$.}
\STATE{Output: biorthogonalization tensor series $\mathcal{V}_{j}$, $\mathcal{W}_{j}$, $j=1,2,...$}
\FOR{$j=1,2,...$}
\STATE{$\alpha_{j}=\langle\mathcal{L}^{2}(\mathcal{V}_{j}), \mathcal{W}_{j} \rangle$}
\STATE{$\overline{\mathcal{V}}_{j+1}=\mathcal{L}(\mathcal{V}_{j})-\alpha_{j}\mathcal{V}_{j}-\beta_{j}\mathcal{V}_{j-1}$}
\STATE{$\overline{\mathcal{W}}_{j+1}=\mathcal{L}^{T}(\mathcal{W}_{j})-\alpha_{j}\mathcal{W}_{j}-\delta_{j}\mathcal{W}_{j-1}$}
\STATE{$\delta_{j+1}=|\langle\overline{\mathcal{W}}_{j+1}, \mathcal{L}(\overline{\mathcal{V}}_{j+1}) \rangle|^{\frac{1}{2}}$}
\STATE{$\beta_{j+1}=\frac{\langle\overline{\mathcal{W}}_{j+1}, \mathcal{L}(\overline{\mathcal{V}}_{j+1}) \rangle}{\delta_{j+1}}$}
\STATE{$\mathcal{V}_{j+1}=\frac{\overline{\mathcal{V}}_{j+1}}{\delta_{j+1}}$}
\STATE{$\mathcal{W}_{j+1}=\frac{\overline{\mathcal{W}}_{j+1}}{\beta_{j+1}}$}
\ENDFOR
\end{algorithmic}
\end{algorithm}

We have the following results for Algorithm \ref{alg0.1}. The proofs of these results are similar to the proof of Proposition 1 in \cite{CJH11} by using the definitions of the inner product \eqref{eq0.24}, \eqref{eq2.5} and linear operator $\mathcal{L}$ in \eqref{loper1} and are omitted.
\begin{proposition}\label{pro0.31}
If Alogrithm \ref{alg0.1} stops at the $m$-th step, then the tensors $\mathcal{V}_{j}$ and $\mathcal{W}_{i} (i,j=1,2,...,m)$  produced Algorithm \ref{alg0.1} are $\mathcal{L}$-biorthogonal, i.e.,
\begin{equation}\label{p1}
\langle\mathcal{W}_{i}, \mathcal{L}(\mathcal{V}_{j})\rangle=\delta_{i,j}, 1\leq i, j\leq m,
\end{equation}
where
\begin{equation}
 \delta_{i,j}=\left\{
 \begin{aligned}
               1, \qquad & & {i=j}, \\
                                0     , \qquad & & {otherwise}. \\
\end{aligned}
\right.
\end{equation}
\end{proposition}
\begin{proposition}\label{pro0.32}
Suppose that $\widetilde{\mathcal{V}}_{m}$ is the ($N+1$)-order tensor with columns $\mathcal{V}_{1}, \mathcal{V}_{2},..., \mathcal{V}_{m}$, and
$\widetilde{\mathcal{W}}_{m}$ is the ($N+1$)-order tensor with the columns $\mathcal{W}_{1}, \mathcal{W}_{2},..., \mathcal{W}_{m}$,
$\widetilde{\mathcal{H}}_{m}$ and $\widetilde{\mathcal{G}}_{m}$ are the ($N+1$)-order tensors with the columns $\mathcal{H}_{j}:=\mathcal{L}(\mathcal{V}_{j})$ and $\mathcal{G}_{j}:=\mathcal{L}^{T}(\mathcal{W}_{j})$ $(j=1,2,...,m)$, respectively.
Then we have
\begin{equation}\label{eq0.381}
\widetilde{\mathcal{H}}_{m}=\widetilde{\mathcal{V}}_{m+1}\times_{(N+1)}\underline{\textbf{T}_{m}^{T}}
\end{equation}
and
\begin{equation}\label{eq0.3101}
\widetilde{\mathcal{G}}_{m}=\widetilde{\mathcal{W}}_{m+1}\times_{(N+1)}\underline{\textbf{T}_{m}},
\end{equation}
where
\begin{equation}
\underline{\textbf{T}_{m}}=
\begin{pmatrix}\textbf{T}_{m}\\
\delta_{m+1}\textbf{e}_{m}^{T}\\
\end{pmatrix}
\end{equation}
with
\begin{equation}
\textbf{T}_{m}=
\begin{pmatrix}
\alpha_{1}& \beta_{2} &  &  &\\
 \delta_{2}& \alpha_{2} &\beta_{3} & & \\
  & \ddots & \ddots & \ddots & \\
  &  &  \delta_{m-1}& \alpha_{m-1} & \beta_{m} \\
  &  &  & \delta_{m} & \alpha_{m}\\
\end{pmatrix}
\end{equation}
being a triangular matrix with its elements generated by Algorithm \ref{alg0.1}.
Moreover, it holds that
\begin{equation}\label{eq0.312}
\widetilde{\mathcal{W}}_{m}\boxtimes^{(N+1)}\widetilde{\mathcal{H}}_{m}=\textbf{E}_{m}
\end{equation}
and
\begin{equation}\label{eq0.313}
\widetilde{\mathcal{W}}_{m}\boxtimes^{(N+1)}\mathcal{L}(\widetilde{\mathcal{H}}_{m})=\textbf{T}_{m},
\end{equation}
where $\textbf{E}_{m}$ denotes the identity matrix with $m$ order.
\end{proposition}

We remark that \eqref{eq0.381} and \eqref{eq0.3101} can be represented as
\begin{equation}\label{eq0.38}
\widetilde{\mathcal{H}}_{m}=\widetilde{\mathcal{V}}_{m}\times_{(N+1)}T_{m}^{T}+\delta_{m+1}\mathcal{Z}_{1}\times_{(N+1)}\textbf{K}_{m}
\end{equation}
and
\begin{equation}\label{eq0.310}
\widetilde{\mathcal{G}}_{m}=\widetilde{\mathcal{W}}_{m}\times_{(N+1)}T_{m}+\beta_{m+1}\mathcal{Z}_{2}\times_{(N+1)}\textbf{K}_{m},
\end{equation}
where $\mathcal{Z}_{1}$ is an ($N+1$)-order tensor with $m$ column tensors $\mathcal{O},...,\mathcal{O},\mathcal{V}_{m+1}$, and $\mathcal{Z}_{2}$ is an ($N+1$)-order tensor with $m$ column tensors $\mathcal{O},...,\mathcal{O},\mathcal{W}_{m+1}$, and $\textbf{K}_{m}$ is an $m\times m$ matrix of the form $\textbf{K}_{m}=[0,...,0,\textbf{e}_{m}]$ with $\textbf{e}_{m}$ being the $m$-th column of $\textbf{E}_{m}$.

With the results above we can present the Lanczos $\mathcal{L}$-Biorthogonalization algorithm in tensor form for solving \eqref{eq1.1}.
For any initial tensor $\mathcal{X}_{0}\in\mathbb{R}^{I_{1}\times I_{2}\times...\times I_{N}}$, let $\mathcal{R}_{0}=\mathcal{D}-\mathcal{L}(\mathcal{X}_{0})$ denote its residual. Let $\mathcal{V}_{1}=\mathcal{R}_{0}/\|\mathcal{R}_{0}\|$ and
\begin{equation}
\mathcal{X}_{m}\in\mathcal{X}_{0}+\mathcal{K}_{m}(\mathcal{L},\mathcal{V}_{1}),
\end{equation}
then
\begin{equation}\label{eq0.316}
\mathcal{R}_{m}=(\mathcal{D}-\mathcal{L}(\mathcal{X}_{m}))\perp\mathcal{L}^{T}(\mathcal{K}_{m}(\mathcal{L}^{T},\mathcal{W}_{1})).
\end{equation}
It is easy to verify that a series of tensors $\{\mathcal{V}_{1},\mathcal{V}_{2},...,\mathcal{V}_{m}\}$ produced via Algorithm \ref{alg0.1} form a basis of $\mathcal{K}_{m}(\mathcal{L},\mathcal{V}_{1})$. Thus we have
\begin{equation}\label{eq0.317}
\mathcal{X}_{m}=\mathcal{X}_{0}+\widetilde{\mathcal{V}}_{m}\overline{\times}_{(N+1)}\textbf{y}_{m},
\end{equation}
where $\textbf{y}_{m}\in\mathbb{R}^{m}$. 
By Eq \eqref{eq0.316} and \eqref{eq0.317}, we have
\begin{equation}
\langle \mathcal{D}-\mathcal{L}(\mathcal{X}_{0}+	\widetilde{\mathcal{V}}_{m}\overline{\times}_{(N+1)}\textbf{y}_{m}), \mathcal{L}^{T}(\widetilde{\mathcal{W}}_{m})\rangle=0.
\end{equation}
A further computation results in
\begin{equation}
\langle \mathcal{L}(\widetilde{\mathcal{H}}_{m})\overline{\times}_{(N+1)}\textbf{y}_{m}, \widetilde{\mathcal{W}}_{m}\rangle
=\langle \mathcal{L}(\mathcal{R}_{0}), \widetilde{\mathcal{W}}_{m}\rangle.
\end{equation}
Through a simple inner product operation and according to Eq.\eqref{eq2.6} and Lemma \ref{lem2.3} we have
\begin{equation}\label{eq0.318}
(\widetilde{\mathcal{W}}_{m}\boxtimes^{(N+1)}\mathcal{L}(\widetilde{\mathcal{H}}_{m}))\textbf{y}_{m}
=\widetilde{\mathcal{W}}_{m}\boxtimes^{(N+1)}\mathcal{L}(\mathcal{R}_{0}).
\end{equation}
Submitting Eq.\eqref{eq0.313} into  Eq.\eqref{eq0.318} results in the tridiagonal system on $y_{m}$:
\begin{equation}\label{eq0.319}
\textbf{T}_{m}\textbf{y}_{m}=\|\mathcal{R}_{0}\|\textbf{e}_{1}.
\end{equation}
Once we compute $y_{m}$ by \eqref{eq0.319}, we get the solution $\mathcal{X}_{m}$ of \eqref{eq1.1} by \eqref{eq0.317}. We summarize this method in Algorithm \ref{alg3.2}, which is called Tensor Lanczos $\mathcal{L}$-Biorthogonalization algorithm (TLB).
\begin{algorithm}
\caption{TLB: A tensor Lanczos $\mathcal{L}$-biorthogonalization Algorithm for solving \eqref{eq1.1}}
\label{alg3.2}
\begin{algorithmic}
\STATE{Choose an initial tensor $\mathcal{X}_{0}$ and compute $\mathcal{R}_{0}=\mathcal{D}-\mathcal{L}(\mathcal{X}_{0})$}.
\STATE{Set $\mathcal{V}_{1}=\frac{\mathcal{R}_{0}}{\|\mathcal{R}_{0}\|}$, choose a tensor $\mathcal{W}_{1}$ such that $\langle\mathcal{L}(\mathcal{V}_{1}), \mathcal{W}_{1}\rangle=1$.}
\FOR{$m=1,2,...$ until convergence}
\STATE{Compute Lanczos $\mathcal{L}$-Biorthogonalization tensors $\mathcal{V}_{1},...,\mathcal{V}_{m}$, $\mathcal{W}_{1},...,\mathcal{W}_{m}$ and $\textbf{T}_{m}$ by Algorithm \ref{alg0.1}.}
\STATE{Compute $\textbf{y}_{m}$ by \eqref{eq0.319}.}
\ENDFOR
\STATE{Compute the solution $\mathcal{X}_{m}$ of \eqref{eq1.1} by \eqref{eq0.317}.}
\end{algorithmic}
\end{algorithm}

We remark Algorithm \ref{alg3.2} have to compute the inverse of $\textbf{T}_{m}$. When $\textbf{T}_{m}$ is of large size, it needs much computation. We present two improved algorithms for Algorithm \ref{alg3.2} in the next section.

\section{The TBiCOR and TCORS Algorithms}\label{sec4}

\subsection{The TBiCOR Algorithm}\label{sec4.1}
In this subsection, we develop an improved algorithm by introducing the $LU$ decomposition to $\textbf{T}_{m}$ in Algorithm \ref{alg3.2}.

Let the $LU$ decomposition of $\textbf{T}_{m}$ be
\begin{equation}\label{eq0.40}
\textbf{T}_{m}=\textbf{L}_{m}\textbf{U}_{m},
\end{equation}
then, according to Lemma \ref{lem2.1}, substituting \eqref{eq0.319} and \eqref{eq0.40} into \eqref{eq0.317} results in
\begin{align}\label{eq0.401}
\mathcal{X}_{m}&=\mathcal{X}_{0}+\widetilde{\mathcal{V}}_{m}\overline{\times}_{(N+1)}\textbf{y}_{m}\nonumber\\
&=\mathcal{X}_{0}+\widetilde{\mathcal{V}}_{m}\overline{\times}_{(N+1)}(\textbf{U}_{m}^{-1}\textbf{L}_{m}^{-1}(\|\mathcal{R}_{0}\|\textbf{e}_{1}))\nonumber\\
&=\mathcal{X}_{0}+\widetilde{\mathcal{P}}_{m}\overline{\times}_{(N+1)}\textbf{z}_{m},
\end{align}
where $\textbf{z}_{m}=\textbf{L}_{m}^{-1}(\|\mathcal{R}_{0}\|\textbf{e}_{1})$ and $\widetilde{\mathcal{P}}_{m}=\widetilde{\mathcal{V}}_{m}\times_{(N+1)}(\textbf{U}_{m}^{-1})^{T}$.

We consider the solution of the system $\mathcal{L}^{T}(\mathcal{X}^{*})=\mathcal{D}^{*}$. The dual approximation $\mathcal{X}_{m}^{*}$ is the  subspace $\mathcal{X}_{0}^{*}+\mathcal{K}_{m}(\mathcal{L}^{T},\mathcal{W}_{1})$ that satisfies
\begin{align*}
(\mathcal{D}^{*}-\mathcal{L}^{T}(\mathcal{X}_{m}^{*}))\perp\mathcal{L}(\mathcal{K}_{m}(\mathcal{L},\mathcal{V}_{1})).
\end{align*}
Set $\mathcal{R}_{0}^{*}=\mathcal{D}^{*}-\mathcal{L}^{T}(\mathcal{X}_{0}^{*})$ and $\mathcal{W}_{1}=\mathcal{R}_{0}^{*}/\|\mathcal{R}_{0}^{*}\|$. If we choose $\mathcal{V}_{1}$ such that $\langle\mathcal{V}_{1}, \mathcal{L}(\mathcal{W}_{1})\rangle=1$, then similar to \eqref{eq0.317}-\eqref{eq0.319}, the solution of the dual system $\mathcal{L}^{T}(\mathcal{X}^{*})=\mathcal{D}^{*}$ can be represented as
\begin{equation}\label{eq0.41}
\mathcal{X}_{m}^{*}=\mathcal{X}_{0}^{*}+\widetilde{\mathcal{W}}_{m}\overline{\times}_{(N+1)}\textbf{y}_{m}^{*},
\end{equation}
where $\textbf{y}_{m}^{*}$ is derived from
\begin{equation}\label{eq0.42}
\textbf{T}_{m}^{T}\textbf{y}_{m}^{*}=\|\mathcal{R}_{0}^{*}\|\textbf{e}_{1}.
\end{equation}
Similar to \eqref{eq0.401}, according to Lemma \ref{lem2.1}, \eqref{eq0.41} together with \eqref{eq0.40} and \eqref{eq0.42} results in
\begin{align*}
\mathcal{X}_{m}^{*}&=\mathcal{X}_{0}^{*}+\widetilde{\mathcal{P}}_{m}^{*}\overline{\times}_{(N+1)}\textbf{z}_{m}^{*},
\end{align*}
where $\widetilde{\mathcal{P}}_{m}^{*}=\widetilde{\mathcal{W}}_{m}\times_{(N+1)}\textbf{L}_{m}^{-1}$, and $\textbf{z}_{m}^{*}=(\textbf{U}_{m}^{T})^{-1}(\|\mathcal{R}_{0}^{*}\|\textbf{e}_{1})$.

\begin{proposition}\label{pro4.1}
Let $\mathcal{R}_{i}=\mathcal{D}-\mathcal{L}(\mathcal{X}_{i})$ and $\mathcal{R}_{i}^{*}=\mathcal{D}^{*}-\mathcal{L}^{T}(\mathcal{X}_{i}^{*})$ are the $i$-th residual tensor and the $i$-th dual residual tensor, respectively, then it holds that
\begin{equation}\label{eq3.20}
\langle \mathcal{L}(\mathcal{R}_{i}), \mathcal{R}_{j}^{*}\rangle=0, (0\leq i\neq j\leq k).
\end{equation}
\end{proposition}
\begin{proof} According to Lemma \ref{lem2.1}, \eqref{eq0.38}, \eqref{eq0.310}, \eqref{eq0.317} and \eqref{eq0.41}, we have
\begin{equation}\label{eq4.2}
\begin{aligned}
\mathcal{R}_{i}&=\mathcal{D}-\mathcal{L}(\mathcal{X}_{0}+\widetilde{\mathcal{V}}_{i}\overline{\times}_{(N+1)}\textbf{y}_{i})\\
&=\mathcal{R}_{0}-\mathcal{L}(\widetilde{\mathcal{V}}_{i})\overline{\times}_{(N+1)}\textbf{y}_{i}\\
&=\mathcal{R}_{0}-\widetilde{\mathcal{V}}_{i}\times_{(N+1)}\textbf{T}_{i}^{T}\overline{\times}_{(N+1)}\textbf{y}_{i}-\delta_{i+1}\mathcal{Z}_{1}\times_{(N+1)}\textbf{K}_{i}\overline{\times}_{(N+1)}\textbf{y}_{i}\\
&=\mathcal{R}_{0}-\widetilde{\mathcal{V}}_{i}\overline{\times}_{(N+1)}(\textbf{T}_{i}\textbf{y}_{i})-\delta_{i+1}\mathcal{Z}_{1}\overline{\times}_{(N+1)}(\textbf{K}_{i}^{T}\textbf{y}_{i})\\
&=-\delta_{i+1}\textbf{e}_{i}^{T}\textbf{y}_{i}\mathcal{V}_{i+1}.
\end{aligned}
\end{equation}
Similarly, we can prove that
\begin{equation}\label{eq4.3}
\mathcal{R}_{j}^{*}=-\beta_{j+1}\textbf{e}_{j}^{T}\textbf{y}_{j}^{*}\mathcal{W}_{j+1}.
\end{equation}
\eqref{eq4.2} and \eqref{eq4.3} together with Proposition \ref{pro0.31} result in \eqref{eq3.20}.
\end{proof}
\begin{proposition}\label{pro4.2}
Let $\mathcal{P}_{i}$ and $\mathcal{P}_{i}^{*} (i=1,...,k)$ are the $i$-th column tensor of $\widetilde{\mathcal{P}}_{k}$ and $\widetilde{\mathcal{P}}_{k}^{*}$, respectively. It holds that
\begin{equation}\label{eq0.499}
\langle \mathcal{L}^{2}(\mathcal{P}_{i}), \mathcal{P}_{j}^{*}\rangle=0 (i,j=1,...,k,i\neq j).
\end{equation}
\end{proposition}
\begin{proof}
According to Lemma \ref{lem2.4} and \eqref{eq0.313}, we have
	\begin{align}
(\widetilde{\mathcal{P}}_{k}^{*}\boxtimes^{(N+1)}\mathcal{L}^{2}(\widetilde{\mathcal{P}}_{k}))_{ij}&=((\widetilde{\mathcal{W}}_{k}\times_{(N+1)}\textbf{L}_{k}^{-1})\boxtimes^{(N+1)}\mathcal{L}^{2}(\widetilde{\mathcal{V}}_{k}\times_{(N+1)}(\textbf{U}_{k}^{-1})^{T}))_{ij}\nonumber\\
&=(\textbf{L}_{k}^{-1}(\widetilde{\mathcal{W}}_{k}\boxtimes^{(N+1)}\mathcal{L}(\widetilde{\mathcal{H}}_{k}))\textbf{U}_{k}^{-1})_{ij}\nonumber\\
&=(\textbf{L}_{k}^{-1}\textbf{T}_{k}\textbf{U}_{k}^{-1})_{ij}\nonumber\\
&=(\textbf{E}_{k})_{ij},\label{eq0.43}
\end{align}
which implies that \eqref{eq0.499} holds.
\end{proof}

For a given initial guess $\mathcal{X}_{0}$, let $\mathcal{R}_{0}=\mathcal{D}-\mathcal{L}(\mathcal{X}_{0})$ and $\mathcal{P}_{0}=\mathcal{R}_{0}$.
Set
\begin{align}	
	\mathcal{X}_{j+1}&=\mathcal{X}_{j}+\alpha_{j}\mathcal{P}_{j},\\
	\mathcal{R}_{j+1}&=\mathcal{R}_{j}-\alpha_{j}\mathcal{L}(\mathcal{P}_{j})\label{eq0.411},\\
	\mathcal{P}_{j+1}&=\mathcal{R}_{j+1}+\beta_{j}\mathcal{P}_{j}\label{eq0.412},\quad j=0,1,....
\end{align}
Similarly, for the dual linear system $\mathcal{L}^{T}(\mathcal{X}^{*})=\mathcal{D}^{*} $, we set
\begin{align}
	\mathcal{R}_{j+1}^{*}&=\mathcal{R}_{j}^{*}-\alpha_{j}\mathcal{L}^{T}(\mathcal{P}_{j}^{*}),\mathcal{R}_{0}^{*}=\mathcal{L}(\mathcal{R}_{0}),\\
	\mathcal{P}_{j+1}^{*}&=\mathcal{R}_{j+1}^{*}+\beta_{j}\mathcal{P}_{j}^{*}\label{eq0.414} \quad for\quad j=0,1,....
\end{align}
Now we determine $\alpha_{j}$ and $\beta_{j}$ in \eqref{eq0.411}-\eqref{eq0.414}. According to \eqref{eq0.411} we have that
\begin{align}
	\langle \mathcal{L}(\mathcal{R}_{j+1}), \mathcal{R}_{j}^{*}\rangle&=\langle \mathcal{L}(\mathcal{R}_{j})-\alpha_{j}\mathcal{L}^{2}(\mathcal{P}_{j}), \mathcal{R}_{j}^{*}\rangle=0,
\end{align}
then by Propositions \ref{pro4.1}, \ref{pro4.2} and \eqref{eq0.411}-\eqref{eq0.414} it holds that
\begin{align}
	\alpha_{j}&=\frac{\langle \mathcal{L}(\mathcal{R}_{j}), \mathcal{R}_{j}^{*}\rangle}{\langle \mathcal{L}^{2}(\mathcal{P}_{j}), \mathcal{R}_{j}^{*}\rangle}
	=\frac{\langle \mathcal{L}(\mathcal{R}_{j}), \mathcal{R}_{j}^{*}\rangle}{\langle \mathcal{L}(\mathcal{P}_{j}), \mathcal{L}^{T}(\mathcal{P}_{j}^{*})\rangle}.\label{eq0.418}
\end{align}
Similarly according to
\begin{align}
	\langle \mathcal{L}^{2}(\mathcal{P}_{j+1}), \mathcal{P}_{j}^{*}\rangle=\langle \mathcal{L}(\mathcal{P}_{j+1}), \mathcal{L}^{T}(\mathcal{P}_{j}^{*})\rangle=0,
\end{align}
we have
\begin{align}
	\beta_{j}&=-\frac{\langle \mathcal{L}(\mathcal{R}_{j+1}), \mathcal{L}^{T}(\mathcal{P}_{j}^{*})\rangle}{\langle \mathcal{L}(\mathcal{P}_{j}), \mathcal{L}^{T}(\mathcal{P}_{j}^{*})\rangle}=\frac{\langle \mathcal{L}(\mathcal{R}_{j+1}), \mathcal{R}_{j+1}^{*}\rangle}{\langle \mathcal{L}(\mathcal{R}_{j}), \mathcal{R}_{j}^{*}\rangle}.\label{eq0.422}
\end{align}

Algorithm \ref{alg3.3} summarizes the biconjugate $\mathcal{L}$-orthogonal residual algorithm in Tensor form for solving \eqref{eq1.1}, which is abbreviated as TBiCOR.

\begin{algorithm}
\caption{TBiCOR: A tensor biconjugate $\mathcal{L}$-orthogonal residual algorithm for solving (\ref{eq1.1})}
\label{alg3.3}
\begin{algorithmic}
\STATE{Compute $\mathcal{R}_{0}=\mathcal{D}-\mathcal{L}(\mathcal{X}_{0})$ ($\mathcal{X}_{0}$ is an initial guess)}
\STATE{Set $\mathcal{R}_{0}^{*}=\mathcal{L}(\mathcal{R}_{0})$}
\STATE{Set $\mathcal{P}_{-1}^{*}=\mathcal{P}_{-1}=0$, $\beta_{-1}=0$}
\FOR{n=0,1,..., until convergence}
\STATE{$\mathcal{P}_{n}=\mathcal{R}_{n}+\beta_{n-1}\mathcal{P}_{n-1}$}
\STATE{$\mathcal{P}_{n}^{*}=\mathcal{R}_{n}^{*}+\beta_{n-1}\mathcal{P}_{n-1}^{*}$}
\STATE{$\mathcal{S}_{n}=\mathcal{L}(\mathcal{P}_{n})$}
\STATE{$\mathcal{S}_{n}^{*}=\mathcal{L}^{T}(\mathcal{P}_{n}^{*})$}
\STATE{$\mathcal{T}_{n}=\mathcal{L}(\mathcal{R}_{n})$}
\STATE{$\alpha_{n}=\frac{\langle \mathcal{R}_{n}^{*}, \mathcal{T}_{n}\rangle}{\langle \mathcal{S}_{n}^{*}, \mathcal{S}_{n}\rangle}$}
\STATE{$\mathcal{X}_{n+1}=\mathcal{X}_{n}+\alpha_{n}\mathcal{P}_{n}$}
\STATE{$\mathcal{R}_{n+1}=\mathcal{R}_{n}-\alpha_{n}\mathcal{S}_{n}$}
\STATE{$\mathcal{R}_{n+1}^{*}=\mathcal{R}_{n}^{*}-\alpha_{n}\mathcal{S}_{n}^{*}$}
\STATE{$\mathcal{T}_{n+1}=\mathcal{L}(\mathcal{R}_{n+1})$}
\STATE{$\beta_{n}=\frac{\langle \mathcal{R}_{n+1}^{*}, \mathcal{T}_{n+1}\rangle}{\langle \mathcal{R}_{n}^{*}, \mathcal{T}_{n}\rangle}$}
\ENDFOR
\end{algorithmic}
\end{algorithm}

We have the following convergence properties on Algorithm \ref{alg3.3}.
\begin{theorem}\label{the3.2}
Assume that the Sylvester tensor equation \eqref{eq1.1} is consistent. For any initial tensor $\mathcal{X}_{0}\in\mathbb{R}^{I_{1}\times I_{2}\times...\times I_{N}}$, Algorithm \ref{alg3.3} converges to an exact solution of \eqref{eq1.1} at most $M=I_{1}\times I_{2}\times...\times I_{N}$ iteration steps in the absence of roundoff errors.
\end{theorem}
\begin{proof} Suppose that $\mathcal{R}_{k}\neq\mathcal{O}(k=0,1,...,M)$ and
\begin{displaymath}
\sum_{k=0}^{M}\lambda_{k}\mathcal{R}_{k}=\mathcal{O}.
\end{displaymath}
According to Proposition \ref{pro4.1}, we have
\begin{align*}
0&=\langle \mathcal{R}_{i}^{*}, \sum_{k=0}^{M}\lambda_{k}\mathcal{L}(\mathcal{R}_{k})\rangle=\sum_{k=0}^{M}\lambda_{k}\langle\mathcal{R}_{i}^{*}, \mathcal{L}(\mathcal{R}_{k})\rangle\\
&=\lambda_{i}\langle \mathcal{R}_{i}^{*}, \mathcal{L}(\mathcal{R}_{i})\rangle, i=0,1,...,M.
\end{align*}
When Algorithm \ref{alg3.3} does not break down,  $\langle \mathcal{R}_{i}^{*}, \mathcal{L}(\mathcal{R}_{i})\rangle\neq0(i=0,1,...,M)$, which leads to $\lambda_{i}=0(i=0,1,...,M)$. This means that $\mathcal{R}_{0}$,$\mathcal{R}_{1}$,...,$\mathcal{R}_{M}$ are linearly independent, while the  dimension of tensor space $\mathbb{R}^{I_{1}\times I_{2}\times...\times I_{N}}$ is $M$. This is a contradiction. Thus Algorithm \ref{alg3.3} converges to an exact solution within $M$ steps.
\end{proof}

\subsection{The TCORS Algorithm}\label{sec4.2}
This subsection presents an improved method on Algorithm \ref{alg3.3} by introducing a squared operator of the residual of $\mathcal{X}_{n}$ produced by Algorithm \ref{alg3.3}. The proposed method is called the conjugate $\mathcal{L}$-orthogonal residual squared algorithm in tensor form, which is abbreviated as TCORS.

\begin{algorithm}
\caption{TCORS: A tensor conjugate $\mathcal{L}$-orthogonal residual squared algorithm for solving (\ref{eq1.1})}
\label{alg3.5}
\begin{algorithmic}
\STATE{Compute $\mathcal{R}_{0}=\mathcal{D}-\mathcal{L}(\mathcal{X}_{0})$; ($\mathcal{X}_{0}$ is an initial guess)}
\STATE{Set $\mathcal{R}_{0}^{*}=\mathcal{L}(\mathcal{R}_{0})$}
\FOR{n=1,2,..., until convergence}
\STATE{$\mathcal{U}_{0}=\mathcal{R}_{0}$, $\widehat{\mathcal{Z}}=\mathcal{L}(\mathcal{U}_{n-1})$; $\rho_{n-1}=\langle \mathcal{R}_{0}^{*}, \widehat{\mathcal{Z}}\rangle$; $\mathcal{Z}_{n-1}=\mathcal{U}_{n-1}$}
\STATE{if $\rho_{n-1}=0$, stop and reset the initial tensor $\mathcal{X}_{0}$.}
\STATE{if $n=1$}
\STATE{$\mathcal{T}_{0}=\mathcal{U}_{0}$; $\mathcal{D}_{0}=\mathcal{T}_{0}$; $\mathcal{C}_{0}=\widehat{\mathcal{Z}}$; $\mathcal{Q}_{0}=\widehat{\mathcal{Z}}$}
\STATE{else}
\STATE{$\beta_{n-2}=\rho_{n-1}/\rho_{n-2}$; $\mathcal{T}_{n-1}=\mathcal{U}_{n-1}+\beta_{n-2}\mathcal{H}_{n-2}$}
\STATE{$\mathcal{D}_{n-1}=\mathcal{Z}_{n-1}+\beta_{n-2}\mathcal{V}_{n-2}$; $\mathcal{C}_{n-1}=\widehat{\mathcal{Z}}+\beta_{n-2}\mathcal{F}_{n-2}$}
\STATE{$\mathcal{Q}_{n-1}=\mathcal{C}_{n-1}+\beta_{n-2}(\mathcal{F}_{n-2}+\beta_{n-2}\mathcal{Q}_{n-2})$}
\STATE{end if}
\STATE{$\widehat{\mathcal{Q}}=\mathcal{L}(\mathcal{Q}_{n-1})$}
\STATE{$\alpha_{n-1}=\rho_{n-1}/\langle \mathcal{R}_{0}^{*}, \widehat{\mathcal{Q}} \rangle$; $\mathcal{H}_{n-1}=\mathcal{T}_{n-1}-\alpha_{n-1}\mathcal{Q}_{n-1}$}
\STATE{$\mathcal{V}_{n-1}=\mathcal{D}_{n-1}-\alpha_{n-1}\mathcal{Q}_{n-1}$; $\mathcal{F}_{n-1}=\mathcal{C}_{n-1}-\alpha_{n-1}\widehat{\mathcal{Q}}$}
\STATE{$\mathcal{X}_{n}=\mathcal{X}_{n-1}+\alpha_{n-1}(2\mathcal{D}_{n-1}-\alpha_{n-1}\mathcal{Q}_{n-1})$}
\STATE{$\mathcal{U}_{n}=\mathcal{U}_{n-1}-\alpha_{n-1}(2\mathcal{C}_{n-1}-\alpha_{n-1}\widehat{\mathcal{Q}})$}
\ENDFOR
\end{algorithmic}
\end{algorithm}

The residual tensor of $\mathcal{X}_{n}$ produced by Algorithm \ref{alg3.3} can be represented as
\begin{equation}\label{eq3.25}
\begin{aligned}
\mathcal{R}_{n}&=a_{0}\mathcal{R}_{0}+a_{1}\mathcal{L}(\mathcal{R}_{0})+a_{2}\mathcal{L}^{2}(\mathcal{R}_{0})+...+a_{n}\mathcal{L}^{n}(\mathcal{R}_{0})\\
&=(a_{0}\mathcal{L}^{0}+a_{1}\mathcal{L}+a_{2}\mathcal{L}^{2}+...+a_{n}\mathcal{L}^{n})\mathcal{R}_{0},
\end{aligned}
\end{equation}
where $a_{i}$ is determined by Algorithm \ref{alg3.3}. Denote $\varphi_{n}(\mathcal{L})=a_{0}\mathcal{L}^{0}+a_{1}\mathcal{L}+a_{2}\mathcal{L}^{2}+...+a_{n}\mathcal{L}^{n}$, then \eqref{eq3.25} can be represented as
\begin{equation}\label{eq3.251}
\mathcal{R}_{n}=\varphi_{n}(\mathcal{L})\mathcal{R}_{0}.
\end{equation}
Similarly, we have
\begin{equation}\label{eq3.2511}
\mathcal{P}_{n}=\phi_{n}(\mathcal{L})\mathcal{R}_{0},
\end{equation}
where $\phi_{n}(\mathcal{L})=b_{0}\mathcal{L}^{0}+b_{1}\mathcal{L}+b_{2}\mathcal{L}^{2}+...+b_{n}\mathcal{L}^{n}$, and $b_{i}$ can be derived by Algorithm \ref{alg3.3}.
For the directions $\mathcal{R}_{n}^{*}$ and $\mathcal{P}_{n}^{*}$ in Algorithm \eqref{alg3.3}, replacing $\mathcal{L}$ in \eqref{eq3.251} and \eqref{eq3.2511} with $\mathcal{L}^{T}$ results in
\begin{align*}
\mathcal{R}_{n}^{*}=\varphi_{n}(\mathcal{L}^{T})\mathcal{R}_{0}^{*}, \mathcal{P}_{n}^{*}=\phi_{n}(\mathcal{L}^{T})\mathcal{R}_{0}^{*}.
\end{align*}
Thus $\alpha_{n}$ in \eqref{eq0.418} and $\beta_{n}$ in \eqref{eq0.422} can be represented as
\begin{align}\label{eq0.54}
\alpha_{n}=\frac{\langle \mathcal{L}(\varphi_{n}(\mathcal{L})\mathcal{R}_{0}), \varphi_{n}(\mathcal{L}^{T})\mathcal{R}_{0}^{*}\rangle}{\langle\mathcal{L}(\phi_{n}(\mathcal{L})\mathcal{R}_{0}), \mathcal{L}^{T}(\phi_{n}(\mathcal{L}^{T})\mathcal{R}_{0}^{*})\rangle}=\frac{\langle\mathcal{L}(\varphi_{n}^{2}(\mathcal{L})\mathcal{R}_{0}), \mathcal{R}_{0}^{*}\rangle}{\langle\mathcal{L}^{2}(\phi_{n}^{2}(\mathcal{L})\mathcal{R}_{0}), \mathcal{R}_{0}^{*}\rangle},
\end{align}
\begin{align}\label{eq0.5510}
	\beta_{n}=\frac{\langle\varphi_{n+1}(\mathcal{L}^{T})\mathcal{R}_{0}^{*}, \mathcal{L}(\varphi_{n+1}(\mathcal{L})\mathcal{R}_{0}) \rangle}{\langle\varphi_{n}(\mathcal{L}^{T})\mathcal{R}_{0}^{*}, \mathcal{L}(\varphi_{n}(\mathcal{L})\mathcal{R}_{0}) \rangle}=\frac{\langle\mathcal{L}(\varphi_{n+1}^{2}(\mathcal{L})\mathcal{R}_{0}), \mathcal{R}_{0}^{*} \rangle}{\langle \mathcal{L}(\varphi_{n}^{2}(\mathcal{L})\mathcal{R}_{0}), \mathcal{R}_{0}^{*}\rangle}.
\end{align}
According to \eqref{eq0.411}-\eqref{eq0.412}, $\varphi_{j}$ and $\phi_{j}$ can be expressed as
\begin{align}
	\varphi_{j+1}(\mathcal{L})=\varphi_{j}(\mathcal{L})-\alpha_{j}\mathcal{L}(\phi_{j}(\mathcal{L})),\label{eq0.56}\\
	\phi_{j+1}(\mathcal{L})=\varphi_{j+1}(\mathcal{L})+\beta_{j}\phi_{j}(\mathcal{L}),\label{eq0.57}
\end{align}
respectively.
Squaring on both sides of \eqref{eq0.56} and \eqref{eq0.57} results in
\begin{align}	\varphi_{j+1}^{2}(\mathcal{L})&=\varphi_{j}^{2}(\mathcal{L})-2\alpha_{j}\mathcal{L}(\phi_{j}(\mathcal{L})\varphi_{j}(\mathcal{L}))+\alpha_{j}^{2}\mathcal{L}^{2}(\phi_{j}^{2}(\mathcal{L})),\label{eq0.58}\\
\phi_{j+1}^{2}(\mathcal{L})&=\varphi_{j+1}^{2}(\mathcal{L})+2\beta_{j}\varphi_{j+1}(\mathcal{L})\phi_{j}(\mathcal{L})+\beta_{j}^{2}\phi_{j}^{2}(\mathcal{L}).\label{eq0.59}
\end{align}

Furthermore, we have
\begin{align}
	\varphi_{j}(\mathcal{L})\phi_{j}(\mathcal{L})&=\varphi_{j}^{2}(\mathcal{L})+\beta_{j-1}\varphi_{j}(\mathcal{L})\phi_{j-1}(\mathcal{L})\label{eq0.510},\\
	\varphi_{j+1}(\mathcal{L})\phi_{j}(\mathcal{L})&=\varphi_{j}^{2}(\mathcal{L})+\beta_{j-1}\varphi_{j}(\mathcal{L})\phi_{j-1}(\mathcal{L})-\alpha_{j}\mathcal{L}(\phi_{j}^{2}(\mathcal{L})).\label{eq0.511}
\end{align}
Taking \eqref{eq0.510} into \eqref{eq0.58} results in
\begin{align}	\varphi_{j+1}^{2}(\mathcal{L})=\varphi_{j}^{2}(\mathcal{L})-\alpha_{j}\mathcal{L}(2\varphi_{j}^{2}(\mathcal{L})+2\beta_{j-1}\varphi_{j}(\mathcal{L})\phi_{j-1}(\mathcal{L})-\alpha_{j}\mathcal{L}(\phi_{j}^{2}(\mathcal{L}))).\label{eq0.512}
\end{align}
Denote
\begin{align}
	\mathcal{U}_{j}&=\varphi_{j}^{2}(\mathcal{L})\mathcal{R}_{0},\\
	\mathcal{Q}_{j}&=\mathcal{L}(\phi_{j}^{2}(\mathcal{L}))\mathcal{R}_{0},\\
	\mathcal{F}_{j}&=\mathcal{L}(\varphi_{j+1}(\mathcal{L})\phi_{j}(\mathcal{L}))\mathcal{R}_{0},
\end{align}
then
\begin{align}
	\mathcal{U}_{j+1}&=\mathcal{U}_{j}-\alpha_{j}(2\mathcal{L}(\mathcal{U}_{j})+2\beta_{j-1}\mathcal{F}_{j-1}-\alpha_{j}\mathcal{L}(\mathcal{Q}_{j})),\label{eq0.5161}\\
	\mathcal{Q}_{j+1}&=\mathcal{L}(\mathcal{U}_{j+1})+2\beta_{j}\mathcal{F}_{j}+\beta_{j}^{2}\mathcal{Q}_{j},\label{eq0.5171}\\
	\mathcal{F}_{j}&=\mathcal{L}(\mathcal{U}_{j})+\beta_{j-1}\mathcal{F}_{j-1}-\alpha_{j}\mathcal{L}(\mathcal{Q}_{j}).\label{eq0.5181}
\end{align}
Denote
\begin{align}
	\mathcal{L}(\mathcal{U}_{j})+\beta_{j-1}\mathcal{F}_{j-1}=\mathcal{C}_{j},
\end{align}
then \eqref{eq0.5161}-\eqref{eq0.5181} can be represented as
\begin{align}
	\mathcal{U}_{j+1}&=\mathcal{U}_{j}-\alpha_{j}(2\mathcal{C}_{j}-\alpha_{j}\mathcal{L}(\mathcal{Q}_{j})),\\
	\mathcal{Q}_{j+1}&=\mathcal{C}_{j+1}+\beta_{j}\mathcal{F}_{j}+\beta_{j}^{2}\mathcal{Q}_{j},\\
	\mathcal{F}_{j}&=\mathcal{C}_{j}-\alpha_{j}\mathcal{L}(\mathcal{Q}_{j}).
\end{align}

Algorithm \ref{alg3.5} summarizes the TCORS algorithm for solving \eqref{eq1.1}. The following results list the convergence of Algorithm \ref{alg3.5}.

\begin{theorem}\label{the3.20}
Assume the Sylvester tensor equation \eqref{eq1.1} is consistent. For any initial tensor $\mathcal{X}_{0}\in\mathbb{R}^{I_{1}\times I_{2}\times...\times I_{N}}$, the iteration solution $\{\mathcal{X}_{n}\}$ produced by Algorithm \ref{alg3.5} converge to an exact solution of \eqref{eq1.1} at most $M=I_{1}\times I_{2}\times...\times I_{N}$ iteration steps without roundoff errors.
\end{theorem}

\begin{proof} The proof of Theorem \ref{the3.20} is similar to that of Theorem \ref{the3.2} by replacing $\mathcal{R}_{k}$ with $\mathcal{U}_{k}$, thus is omitted.
\end{proof}

\section{Preconditioned BiCOR and TCORSs Algorithms}\label{sec5}
This section presents two preconditioned methods based on Algorithms \ref{alg3.3}-\ref{alg3.5} for solving Eq.\eqref{eq0.61}.

Using the definition of the Kronecker product in \cite {JL13}, one can transform Eq.\eqref{eq1.1} to its equivalent linear system
\begin{equation}\label{eq0.61}
	\textbf{A}\textbf{x}=\textbf{b},
\end{equation}
where $\textbf{A}=\textbf{E}_{I_{N}}\otimes\dots\otimes\textbf{E}_{I_{2}}\otimes\textbf{A}_{1}+\dots+\textbf{A}_{N}\otimes\textbf{E}_{I_{N-1}}\otimes\dots\otimes\textbf{E}_{I_{1}}$, '$\otimes$' denotes the Kronecker product, $\textbf{x}=vec(\mathcal{X})$, $\textbf{b}=vec(\mathcal{D})$. We refer to \cite{TB09} for more details.

\begin{algorithm}
\caption{PTBiCOR: A preconditioned tensor biconjugate $\tilde{\mathcal{L}}$-orthogonal residual algorithm for solving (\ref{eq1.1})}
\label{alg6}
\begin{algorithmic}
\STATE{Compute matrices $\textbf{Q}_{i}(i=1\dots N)$ and $\tilde{\mathcal{D}}=\mathcal{D}\times_{1}\textbf{Q}_{N}^{-1}\times_{2}\dots\times_{N}\textbf{Q}_{1}^{-1}$}.
\STATE{Replace $\mathcal{L}$, $\mathcal{L}^{T}$ in algorithm \ref{alg3.3} with $\tilde{\mathcal{L}}$, $\tilde{\mathcal{L}}^{T}$, $\tilde{\mathcal{L}}(\mathcal{X})=\mathcal{X}\times_{1}(\textbf{Q}_{N}^{-1}\textbf{A}_{1})\times_{2}\dots\times_{N}\textbf{Q}_{1}^{-1}+\dots+\mathcal{X}\times_{1}\textbf{Q}_{N}^{-1}\times_{2}\dots\times_{N}(\textbf{Q}_{1}^{-1}\textbf{A}_{N})$ and $\tilde{\mathcal{L}}^{T}(\mathcal{X})=\mathcal{X}\times_{1}(\textbf{Q}_{N}^{-1}\textbf{A}_{1})^{T}\times_{2}\dots\times_{N}(\textbf{Q}_{1}^{-1})^{T}+\dots+\mathcal{X}\times_{1}(\textbf{Q}_{N}^{-1})^{T}\times_{2}\dots\times_{N}(\textbf{Q}_{1}^{-1}\textbf{A}_{N})^{T}$}.
\STATE{Compute $\mathcal{R}_{0}=\tilde{\mathcal{D}}-\tilde{\mathcal{L}}(\mathcal{X}_{0})$ ($\mathcal{X}_{0}$ is an initial guess)}
\STATE{Set $\mathcal{R}_{0}^{*}=\tilde{\mathcal{L}}(\mathcal{R}_{0})$}
\STATE{Set $\mathcal{P}_{-1}^{*}=\mathcal{P}_{-1}=0$, $\beta_{-1}=0$}
\FOR{n=0,1,..., until convergence}
\STATE{$\mathcal{P}_{n}=\mathcal{R}_{n}+\beta_{n-1}\mathcal{P}_{n-1}$}
\STATE{$\mathcal{P}_{n}^{*}=\mathcal{R}_{n}^{*}+\beta_{n-1}\mathcal{P}_{n-1}^{*}$}
\STATE{$\mathcal{S}_{n}=\tilde{\mathcal{L}}(\mathcal{P}_{n})$}
\STATE{$\mathcal{S}_{n}^{*}=\tilde{\mathcal{L}}^{T}(\mathcal{P}_{n}^{*})$}
\STATE{$\mathcal{T}_{n}=\tilde{\mathcal{L}}(\mathcal{R}_{n})$}
\STATE{$\alpha_{n}=\frac{\langle \mathcal{R}_{n}^{*}, \mathcal{T}_{n}\rangle}{\langle \mathcal{S}_{n}^{*}, \mathcal{S}_{n}\rangle}$}
\STATE{$\mathcal{X}_{n+1}=\mathcal{X}_{n}+\alpha_{n}\mathcal{P}_{n}$}
\STATE{$\mathcal{R}_{n+1}=\mathcal{R}_{n}-\alpha_{n}\mathcal{S}_{n}$}
\STATE{$\mathcal{R}_{n+1}^{*}=\mathcal{R}_{n}^{*}-\alpha_{n}\mathcal{S}_{n}^{*}$}
\STATE{$\mathcal{T}_{n+1}=\tilde{\mathcal{L}}(\mathcal{R}_{n+1})$}
\STATE{$\beta_{n}=\frac{\langle \mathcal{R}_{n+1}^{*}, \mathcal{T}_{n+1}\rangle}{\langle \mathcal{R}_{n}^{*}, \mathcal{T}_{n}\rangle}$}
\ENDFOR
\end{algorithmic}
\end{algorithm}
We are interested in constructing a preconditioner \textbf{M} that transforms Eq.\eqref{eq1.1} to a new system
\begin{equation}\label{eq0.62}
	\textbf{M}\textbf{A}\textbf{x}=\textbf{M}\textbf{b},
\end{equation}
which has the same solution with Eq.\eqref{eq0.61} and has better spectral properties than Eq.\eqref{eq0.61} does. In particular, if $\textbf{M}$ is a good approximation of $\textbf{A}^{-1}$, then Eq.\eqref{eq0.62} can be solved more effectively than Eq.\eqref{eq0.61}. Using the nearest Kronecker product (NKP) in \cite{VP93}, Chen and Lu \cite{CL12} presented an efficient preconditioner for solving Eq.\eqref{eq0.61} based on GMRES in tensor form, which is abbreviated as preconditioned GMRES (PGMRES) later. Zhang and Wang in \cite{ZW21} gave a preconditioned BiCG (PBiCG) and a preconditioned BiCR (PBiCR) based on NKP in \cite{VP93}.
The preconditioner based on NKP approximates $\textbf{A}^{-1}$ by $\textbf{Q}_{1}^{-1}\otimes\textbf{Q}_{2}^{-1}\otimes\dots\otimes\textbf{Q}_{N}^{-1}$ with
\begin{equation}\label{eq0.620}
\left\{
             \begin{array}{lr}
             \textbf{Q}_{1}\approx a_{11}\textbf{A}_{N}+a_{12}\textbf{E}_{I_{N}},  \\
             \textbf{Q}_{2}\approx a_{21}\textbf{A}_{N-1}+a_{22}\textbf{E}_{I_{N-1}}, \\
             \vdots\\
             \textbf{Q}_{N}\approx a_{N1}\textbf{A}_{1}+a_{N2}\textbf{E}_{1},
             \end{array}
\right.
\end{equation}
where the optimal parameters $a_{ij}$ in \eqref{eq0.620} can be computed by using the nonlinear optimization software, such as \emph{fminsearch} in MATLAB.

Introducing the preconditioner based on NKP to Algorithms \ref{alg3.3} and \ref{alg3.5}, we get our preconditioned TBiCOR (PTLB) algorithm and  preconditioned TCORS (TCORS) algorithm for solving Eq \eqref{eq1.1}, which are summarized in Algorithms \ref{alg6} and \ref{alg7}, respectively.

\begin{algorithm}
\caption{PTCORS: A preconditioned tensor conjugate $\tilde{\mathcal{L}}$-orthogonal residual squared algorithm for solving (\ref{eq1.1})}
\label{alg7}
\begin{algorithmic}
\STATE{Compute matrices $\textbf{Q}_{i}(i=1\dots N)$ and $\tilde{\mathcal{D}}=\mathcal{D}\times_{1}\textbf{Q}_{N}^{-1}\times_{2}\dots\times_{N}\textbf{Q}_{1}^{-1}$}.
\STATE{Replace $\mathcal{L}$ in algorithm \ref{alg3.5} with $\tilde{\mathcal{L}}$, $\tilde{\mathcal{L}}(\mathcal{X})=\mathcal{X}\times_{1}(\textbf{Q}_{N}^{-1}\textbf{A}_{1})\times_{2}\dots\times_{N}\textbf{Q}_{1}^{-1}+\dots+\mathcal{X}\times_{1}\textbf{Q}_{N}^{-1}\times_{2}\dots\times_{N}(\textbf{Q}_{1}^{-1}\textbf{A}_{N})$}.
\STATE{Compute $\mathcal{R}_{0}=\tilde{\mathcal{D}}-\tilde{\mathcal{L}}(\mathcal{X}_{0})$; ($\mathcal{X}_{0}$ is an initial guess)}
\STATE{Set $\mathcal{R}_{0}^{*}=\tilde{\mathcal{L}}(\mathcal{R}_{0})$}
\FOR{n=1,2,..., until convergence}
\STATE{$\mathcal{U}_{0}=\mathcal{R}_{0}$, $\widehat{\mathcal{Z}}=\tilde{\mathcal{L}}(\mathcal{U}_{n-1})$; $\rho_{n-1}=\langle \mathcal{R}_{0}^{*}, \widehat{\mathcal{Z}}\rangle$; $\mathcal{Z}_{n-1}=\mathcal{U}_{n-1}$}
\STATE{if $\rho_{n-1}=0$, stop and reset the initial tensor $\mathcal{X}_{0}$.}
\STATE{if $n=1$}
\STATE{$\mathcal{T}_{0}=\mathcal{U}_{0}$; $\mathcal{D}_{0}=\mathcal{T}_{0}$; $\mathcal{C}_{0}=\widehat{\mathcal{Z}}$; $\mathcal{Q}_{0}=\widehat{\mathcal{Z}}$}
\STATE{else}
\STATE{$\beta_{n-2}=\rho_{n-1}/\rho_{n-2}$; $\mathcal{T}_{n-1}=\mathcal{U}_{n-1}+\beta_{n-2}\mathcal{H}_{n-2}$}
\STATE{$\mathcal{D}_{n-1}=\mathcal{Z}_{n-1}+\beta_{n-2}\mathcal{V}_{n-2}$; $\mathcal{C}_{n-1}=\widehat{\mathcal{Z}}+\beta_{n-2}\mathcal{F}_{n-2}$}
\STATE{$\mathcal{Q}_{n-1}=\mathcal{C}_{n-1}+\beta_{n-2}(\mathcal{F}_{n-2}+\beta_{n-2}\mathcal{Q}_{n-2})$}
\STATE{end if}
\STATE{$\widehat{\mathcal{Q}}=\tilde{\mathcal{L}}(\mathcal{Q}_{n-1})$}
\STATE{$\alpha_{n-1}=\rho_{n-1}/\langle \mathcal{R}_{0}^{*}, \widehat{\mathcal{Q}} \rangle$; $\mathcal{H}_{n-1}=\mathcal{T}_{n-1}-\alpha_{n-1}\mathcal{Q}_{n-1}$}
\STATE{$\mathcal{V}_{n-1}=\mathcal{D}_{n-1}-\alpha_{n-1}\mathcal{Q}_{n-1}$; $\mathcal{F}_{n-1}=\mathcal{C}_{n-1}-\alpha_{n-1}\widehat{\mathcal{Q}}$}
\STATE{$\mathcal{X}_{n}=\mathcal{X}_{n-1}+\alpha_{n-1}(2\mathcal{D}_{n-1}-\alpha_{n-1}\mathcal{Q}_{n-1})$}
\STATE{$\mathcal{U}_{n}=\mathcal{U}_{n-1}-\alpha_{n-1}(2\mathcal{C}_{n-1}-\alpha_{n-1}\widehat{\mathcal{Q}})$}
\ENDFOR
\end{algorithmic}
\end{algorithm}

We only give the convergence of Algorithm \ref{alg6}. Similarly we can obtain the convergence of Algorithm \ref{alg7}, thus omit it.
\begin{theorem}\label{the0.4}
Let $\{\mathcal{R}_{i}^{*}\}$, $\{\mathcal{R}_{i}^{*}\}$, $\{\mathcal{P}_{i}\}$ and $\{\mathcal{P}_{i}^{*}\}(i=0,1,...,k)$ be the iterative sequences given by Algorithm \ref{alg6}, then we have
\begin{equation}\label{eq0.63}
\langle \tilde{\mathcal{L}}(\mathcal{R}_{i}), \mathcal{R}_{j}^{*}\rangle=0
\end{equation}
and
\begin{equation}\label{eq0.64}
\langle \tilde{\mathcal{L}}(\mathcal{P}_{i}), \tilde{\mathcal{L}}^{T}(\mathcal{P}_{j}^{*})\rangle=0 (i,j=0,1,...,k,i\neq j).
\end{equation}
\end{theorem}
\begin{proof}
The proof is very similar to those of Propositions \eqref{pro4.1} and \eqref{pro4.2} with the operator $\mathcal{L}$ being replaced by $\tilde{\mathcal{L}}$ in Algorithm \ref{alg6}, and is omitted.
\end{proof}

\begin{theorem}\label{the0.5}
Assume that the Sylvester tensor equation \eqref{eq1.1} is consistent. For any initial tensor $\mathcal{X}_{0}\in\mathbb{R}^{I_{1}\times I_{2}\times...\times I_{N}}$, Algorithm \ref{alg6} converges to an exact solution of \eqref{eq1.1} at most $M=I_{1}\times I_{2}\times...\times I_{N}$ iteration steps in the absence of roundoff errors.
\end{theorem}
\begin{proof}The proof is similar to that of Theorem \ref{the3.2} by replacing $\mathcal{L}$ with $\tilde{\mathcal{L}}$ and is omitted.
\end{proof}

We can also obtain a preconditioned TLB (PTLB) by introducing the NKP preconditioner in \cite{CL12} to Algorithm \ref{alg3.2}, which is listed in Algorithm \ref{alg5}.

\begin{algorithm}
\caption{PTLB: A preconditioned tensor Lanczos $\tilde{\mathcal{L}}$-biorthogonalization Algorithm for solving \eqref{eq1.1}}
\label{alg5}
\begin{algorithmic}
\STATE{Compute matrices $\textbf{Q}_{i}(i=1\dots N)$ and $\tilde{\mathcal{D}}=\mathcal{D}\times_{1}\textbf{Q}_{N}^{-1}\times_{2}\dots\times_{N}\textbf{Q}_{1}^{-1}$}.
\STATE{Replace $\mathcal{L}$, $\mathcal{L}^{T}$ in algorithm \ref{alg0.1}, \ref{alg3.2} with $\tilde{\mathcal{L}}$, $\tilde{\mathcal{L}}^{T}$}, $\tilde{\mathcal{L}}(\mathcal{X})=\mathcal{X}\times_{1}(\textbf{Q}_{N}^{-1}\textbf{A}_{1})\times_{2}\dots\times_{N}\textbf{Q}_{1}^{-1}+\dots+\mathcal{X}\times_{1}\textbf{Q}_{N}^{-1}\times_{2}\dots\times_{N}(\textbf{Q}_{1}^{-1}\textbf{A}_{N})$ and $\tilde{\mathcal{L}}^{T}(\mathcal{X})=\mathcal{X}\times_{1}(\textbf{Q}_{N}^{-1}\textbf{A}_{1})^{T}\times_{2}\dots\times_{N}(\textbf{Q}_{1}^{-1})^{T}+\dots+\mathcal{X}\times_{1}(\textbf{Q}_{N}^{-1})^{T}\times_{2}\dots\times_{N}(\textbf{Q}_{1}^{-1}\textbf{A}_{N})^{T}$.
\STATE{Choose an initial tensor $\mathcal{X}_{0}$ and compute $\mathcal{R}_{0}=\tilde{\mathcal{D}}-\tilde{\mathcal{L}}(\mathcal{X}_{0})$}.
\STATE{Set $\mathcal{V}_{1}=\frac{\mathcal{R}_{0}}{\|\mathcal{R}_{0}\|}$, choose a tensor $\mathcal{W}_{1}$ such that $\langle\tilde{\mathcal{L}}(\mathcal{V}_{1}), \mathcal{W}_{1}\rangle=1$.}
\FOR{$m=1,2,...$ until convergence}
\STATE{Compute Lanczos $\tilde{\mathcal{L}}$-Biorthogonalization tensors $\mathcal{V}_{1},...,\mathcal{V}_{m}$, $\mathcal{W}_{1},...,\mathcal{W}_{m}$ and $\textbf{T}_{m}$ by Algorithm \ref{alg0.1}.}
\STATE{Compute $\textbf{y}_{m}$ by \eqref{eq0.319}.}
\ENDFOR
\STATE{Compute the solution $\mathcal{X}_{m}$ of \eqref{eq1.1} by \eqref{eq0.317}.}
\end{algorithmic}
\end{algorithm}

\section{Numerical Experiments}\label{sec6}
In this section, we show several numerical examples to illustrate Algorithms \ref{alg3.2}--\ref{alg5} and compare them with CGLS in \cite{HM20}, MCG in \cite{LM20}, preconditioned GMRES (PGMRES) in \cite{CL12}, preconditioned BiCG (PBiCG) and preconditioned BiCR (PBiCR) in \cite{ZW21}. All experiments are implemented on a computer with macOS Big Sur 11.1 and 8G memory. The MATLAB R2018a (9.4.0) is used to run all examples. All algorithms are stopped when the relative error $r_{k}=\|\mathcal{X}_{k}-\mathcal{X}^{*}\|/\|\mathcal{X}^{*}\|<10^{-10}$, where $\mathcal{X}^{*}$ is assumed to be an exact solution of \eqref{eq1.1}.
\begin{example}\label{ex4.1}
In this example, we consider the Poisson equation in $d$-dimensional space \cite{JL13}
\begin{equation}\nonumber
\left\{
             \begin{array}{lr}
             -\triangle u=f,\quad  &in\quad\Omega=(0,1)^{d},\\
             u=0,\quad  &on\quad\partial\Omega.
             \end{array}
\right.
\end{equation}
A finite difference discretization leads to the Sylvester tensor equation \eqref{eq1.1}, where $\textbf{A}_{i}\in\mathbb{R}^{10\times10}( i=1,2,...,d)$ are
\begin{equation}\label{eq0.71}
\textbf{A}_{i}=\frac{1}{h^{2}}\begin{pmatrix}2& -1 &  &  &\\
 -1& 2 &-1 & & \\
  & \ddots & \ddots & \ddots & \\
  &  &  -1& 2 &-1 \\
  &  &  & -1 & 2\\
  \end{pmatrix}_{10\times10}
\end{equation}
with the mesh-width $h=\frac{1}{11}$.
\end{example}

\begin{figure}[htbp]
  \begin{center}
    \includegraphics[scale=0.5]{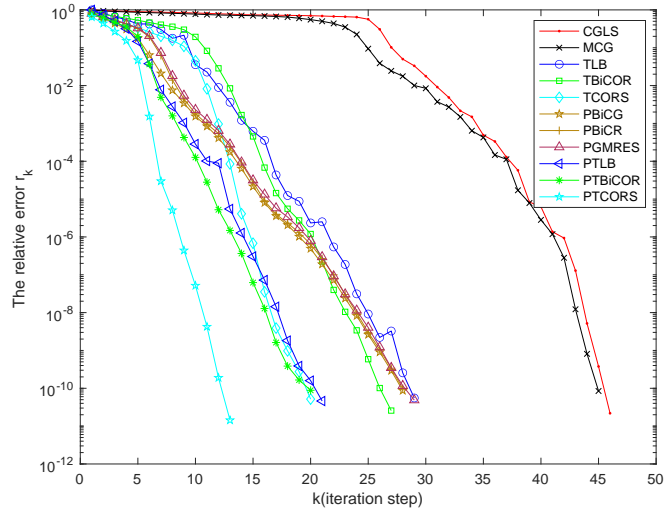}
    \caption{Plot of $r_{k}$ for Example \ref{ex4.1}.}\label{figa}
  \end{center}
\end{figure}

We set $d=3$ and let the initial tensor $\mathcal{X}_{0}=\mathcal{O}$. The right-hand side $\mathcal{D}$ of \eqref{eq1.1} is constructed by \eqref{eq1.1} with the exact solution $\mathcal{X}^{*}$ of \eqref{eq1.3} derived by the MATLAB command $tenones(10,10,10)$ in \cite{WG12}. Algorithms \ref{alg3.2}-\ref{alg7} are used to solve \eqref{eq1.1} with the matrices $A_{i}$ in \eqref{eq0.71}. These methods are compared with CGLS in \cite{HM20} and MCG in \cite{LM20}, respectively.

Figure \ref{figa} shows the convergence of the relative error $r_{k}$ versus the number of iterations for all methods. From Figure \ref{figa}, we can see that our preconditioned Algorithms \ref{alg6}-\ref{alg5} present better convergence than Algorithms \ref{alg3.2}-\ref{alg3.5} without preconditioning, PGMRES in \cite{CL12}, PBiCG and PBiCR in \cite{ZW21}. While Algorithms \ref{alg3.2}-\ref{alg3.5} can compare with PGMRES, PBiCG and PBiCR, and are better than CGLS \cite{HM20} and MCG \cite{LM20}. Algorithm \ref{alg7} converges fastest among all algorithms. Algorithm \ref{alg6} converges the second fastest among all algorithms.

\begin{example}\label{ex4.2}
Consider the convection-diffusion equation in \cite{JL13,HL10}
\begin{equation}\nonumber
\left\{
 \begin{array}{lr}
             v\Delta u+c^{T}\nabla u=f,\quad &in\quad\Omega=[0,1]^{N},\\
             u=0,\quad &on\quad \partial\Omega.
 \end{array}
\right.
\end{equation}
A standard finite difference discretization on equidistant nodes combined with the second order convergent scheme \cite{L04,DC10} for the convection term leads to the linear system \eqref{eq1.1} with
\begin{equation}\label{eq5.5}
\textbf{A}_{n}=\frac{v}{h^{2}}\begin{pmatrix}2& -1 &  &  &\\
 -1& 2 &-1 & & \\
  & \ddots & \ddots & \ddots & \\
  &  &  -1& 2 &-1 \\
  &  &  & -1 & 2\\
  \end{pmatrix}+
  \frac{c_{n}}{4h}\begin{pmatrix}3& -5 & 1 &  &\\
 1& 3 &-5 &1 & \\
  & \ddots & \ddots & \ddots &1 \\
  &  &  1& 3 &-5 \\
  &  &  & 1 & 3\\
  \end{pmatrix}_{p\times p},
\end{equation}
where $n=1,2,...,N$ and the mesh-size $h=\frac{1}{p+1}$.
\end{example}

\begin{table}[htbp]
{\footnotesize
  \caption{Comparison of the running time, total iteration number and the corresponding relative error for different method with different parameters when the criterion is satisfied for Example \ref{ex4.2}}  \label{tab:foo}
\begin{center}
\setlength{\tabcolsep}{1mm}{
  \begin{tabular}{|c|c|c|c|c|c|c|c|c|} \hline
    & Methods & time(s) & TIN & $r_{TIN}$ & Methods & time(s) & TIN & $r_{TIN}$\\ \hline
    $v=1 $ &CGLS &0.422117 &131 &9.7291e-11&PBiCG&0.407984  &27& 5.7115e-11  \\
    $c_{1}=1$ &MCG&0.314574  &130& 9.1914e-11 &PBiCR &0.333801 &27 &6.2653e-11\\
    $c_{2}=1$ & & & & &PGMRES &1.195048 &26 & 8.4492e-11\\
    $c_{3}=1 $ &TLB&0.439216  &48& 8.0006e-11 &PTLB &0.413785 &25 &4.4981e-11\\
    $$ &TBiCOR&0.379316  &48& 1.2006e-11 &PTBiCOR &0.283238 &24 &4.5384e-11  \\
 &TCORS&  0.235870 &32& 7.9107e-11 &PTCORS &0.202668 &15 &7.9490e-12  \\ \hline
$v=0.1$ &CGLS &0.420019 &142& 9.7422e-11 &PBiCG&0.327606   &39 &7.7596e-11  \\
$c_{1}=1 $&MCG&0.349199   &141& 9.1934e-11 &PBiCR &0.365583 &39 &7.7732e-11\\
$c_{2}=1$ & & & & &PGMRES & 2.643044 &37 &7.6420e-11\\
$c_{3}=1 $ &TLB&0.547317  &57& 2.2617e-11  &PTLB &0.452652 &24 &6.9294e-11\\
$ $ &TBiCOR&0.275463  &51& 6.4086e-11 &PTBiCOR &0.240380 &22 &9.2513e-11 \\
 &TCORS&0.199584   &30& 3.6561e-11 &PTCORS &0.167069 &13 &1.5901e-11 \\ \hline
$v=0.01$ &CGLS &0.400485 &137& 7.8655e-11 &PBiCG&0.345495   &23 &5.9009e-12  \\
$c_{1}=1$ &MCG&0.314983   &136& 7.9150e-11 &PBiCR &0.420298 &23 &5.1618e-12\\
$c_{2}=1$ & & & & &PGMRES &0.716276 &20 &2.9726e-11\\
$c_{3}=1 $ &TLB&  0.490665 &53& 4.5843e-11 &PTLB &0.412092 &24 &1.3665e-12 \\
$ $ &TBiCOR&0.431553  &49& 2.6997e-11 &PTBiCOR &0.271265 &22 &6.3548e-11 \\
  &TCORS&0.289261  &29& 4.6591e-11  &PTCORS &0.182132 &14 &1.4634e-11 \\ \hline
$v=1$ &CGLS &0.619692 &231& 9.8161e-11 &PBiCG&0.302540   &24& 9.6344e-11\\
$c_{1}=1 $&MCG&0.479290   &228& 9.5604e-11 &PBiCR &0.332252 &25 &1.8973e-11\\
$c_{2}=2$ & & & & &PGMRES &0.969060 &23 &8.3597e-11\\
$c_{3}=3 $ &TLB&0.574761  &60& 6.1992e-11  &PTLB &0.463983 &25 &6.7805e-11\\
$ $  &TBiCOR& 0.359788  &59& 6.2755e-11  &PTBiCOR &0.294740 &25 &3.1182e-11\\
  &TCORS&0.259970   &33& 9.4662e-11 &PTCORS &0.204611 &15 &1.2034e-11\\ \hline
$v=0.1$ &CGLS &0.661334 &234& 8.5041e-11 &PBiCG&0.321293   &26&  7.8704e-12 \\
$c_{1}=1 $&MCG&0.508388   &231& 9.1198e-11 &PBiCR &0.435445 &24 &7.5118e-11\\
$c_{2}=2$ & & & & &PGMRES &0.885821 &22 &7.6295e-11\\
$c_{3}=3 $ &TLB& 0.500843  &53& 7.0480e-12  &PTLB &0.360964 &22 &1.1481e-12\\
$ $ &TBiCOR& 0.275404  &48& 7.3762e-11 &PTBiCOR &0.244606 &20 &5.4227e-11 \\
 &TCORS& 0.229283 &28& 1.0059e-12  &PTCORS &0.146302 &12 & 1.2266e-11  \\ \hline
$v=0.01$ &CGLS &0.658699 &240& 8.7261e-11  &PBiCG &0.321684 &39 &9.3669e-11\\
$c_{1}=1$ &MCG&0.517902  &236& 9.5783e-11 &PBiCR &0.356874 &38 &9.4120e-11\\
$c_{2}=2$ & & & & &PGMRES &2.472089 &36 &3.8573e-11\\
$c_{3}=3 $ &TLB&0.531822  &55& 6.9508e-11  &PTLB &0.491639 &29 &8.5332e-12\\
$ $  &TBiCOR&0.316190  &54& 5.6945e-11 &PTBiCOR &0.268896 &28 &8.7378e-11\\
$ $ &TCORS&0.203383  &30& 6.8170e-12 &PTCORS &0.164989 &16 &2.5240e-12   \\  \hline
  \end{tabular}}
\end{center}
}
\end{table}

We consider the case when $N=3$ and $p=10$. The right-hand side $\mathcal{D}$ is constructed by \eqref{eq1.1} with the exact solution $\mathcal{X}^{*}$ of \eqref{eq1.3} produced by the MATLAB commend $tenones(10,10,10)$ in \cite{WG12}.

\begin{figure}[htbp]
  \begin{center}
    \includegraphics[scale=0.5]{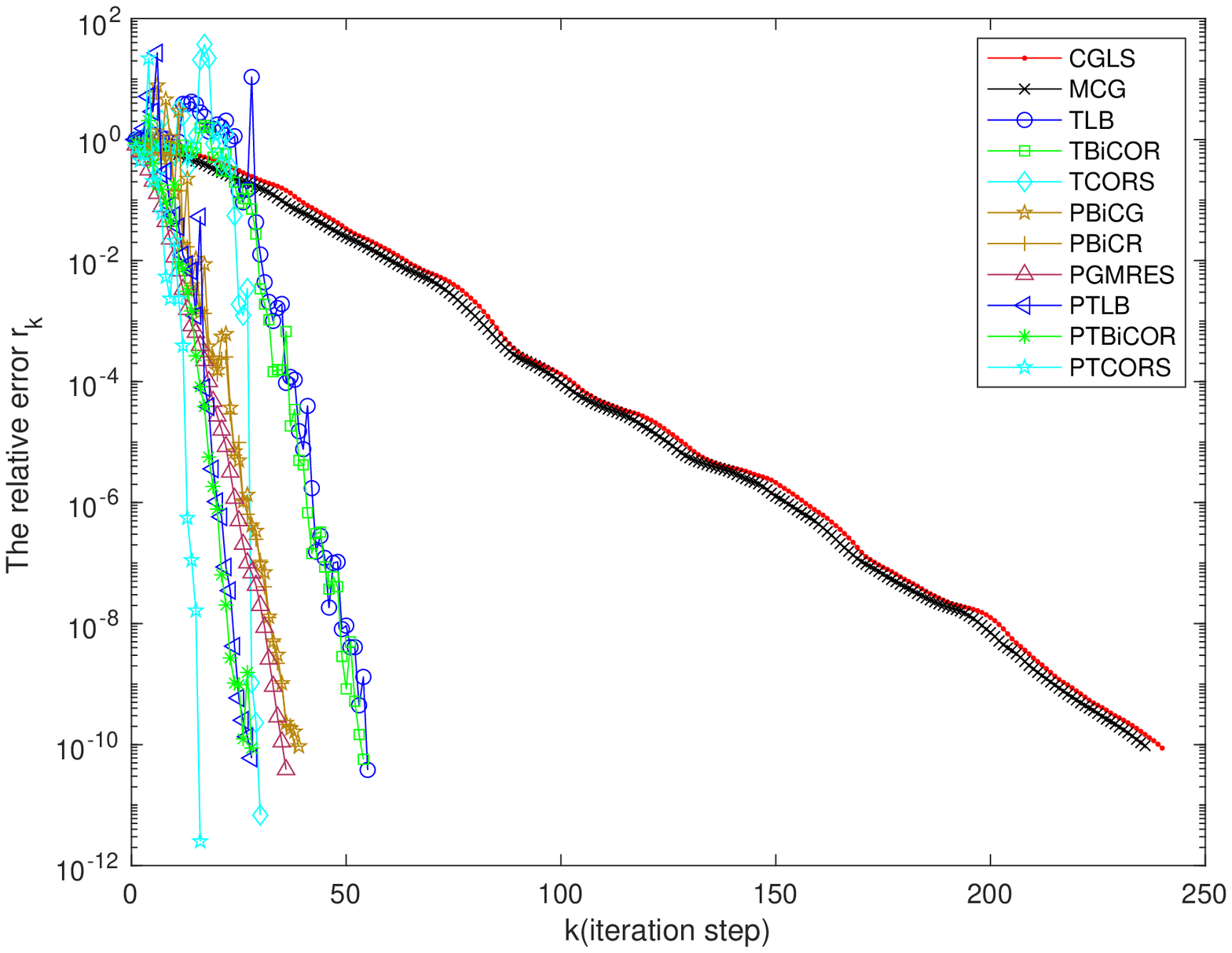}
    \caption{Plot of $r_{k}$ for Example \ref{ex4.2} when $v=0.01,c_{1}=1,c_{2}=2,c_{3}=3$.}\label{figb}
  \end{center}
\end{figure}

Let the initial solution $\mathcal{X}_{0}$ be a tensor with each element being zero. Algorithms \ref{alg3.2}-\ref{alg5} are used to solve \eqref{eq1.1} with $A_{i}$ given in \eqref{eq5.5}. These methods are compared with CGLS \cite{HM20}, MCG \cite{LM20}, PGMRES in \cite{CL12}, PBiCG and PBiCR in \cite{ZW21}.

Table \ref{tab:foo} displays the running time, total iteration number (TIN) and relative error of different method with different parameters $v=1,0.1,0.001$ and $c_i$.
Figure \ref{figb} shows the convergence of the relative error $r_{k}$ for each method with the parameters $v=0.01$, $c_{1}=1$, $c_{2}=2$ and $c_{3}=3$.

Table \ref{tab:foo} shows that, when the stop criterion is satisfied, preconditioned Algorithms \ref{alg6}-\ref{alg7} require less CPU time and iterations than Algorithms \ref{alg3.2}-\ref{alg3.5}, PGMRES, PBiCG and PBiCR. In most cases Algorithms \ref{alg3.2}-\ref{alg3.5} requires much less CPU time but more iterations than PGMRES, PBiCG and PBiCR, and are better than CGLS \cite{HM20} and MCG \cite{LM20} both in CPU time and the number of iterations. Algorithm \ref{alg7} requires the minimal CPU time and iterations among all methods.
Figure \ref{figb} shows similar results to that in Figure \ref{figa}.

\begin{example}\label{ex4.3}
We consider the Sylvester tensor equation \eqref{eq1.1} with the coefficient matrices $\textbf{A}_{i}, i=1,2,3$, which comes from the discretization of the operator
\begin{equation}
Lu:=\triangle u-e^{xy}\frac{\partial u}{\partial x}+sin(xy)\frac{\partial u}{\partial y}+y^2-x^2
\end{equation}
on the unit square $[0,1]\times[0,1]$ with homogeneous Dirichlet boundary conditions. We use the MATLAB command \textit{fdm\_2d\_matrix} in the Lyapack package \cite{T} to generate matrices $A_i$:
\begin{align}
\textbf{A}_{i}=fdm\_2d\_matrix(I_i,e^{xy},sin(xy),y^2-x^2),\label{eq5.6},
\end{align}
where $I_i=1+i, i=1,2,3$.
We construct $\mathcal{D}$ by \eqref{eq1.1} with the exact solution $\mathcal{X}^{*}$ of \eqref{eq1.3} produced by the MATLAB commend $tenones(4,9,16)$ in \cite{WG12}.
\end{example}

\begin{figure}[htbp]
  \begin{center}
    \includegraphics[scale=0.5]{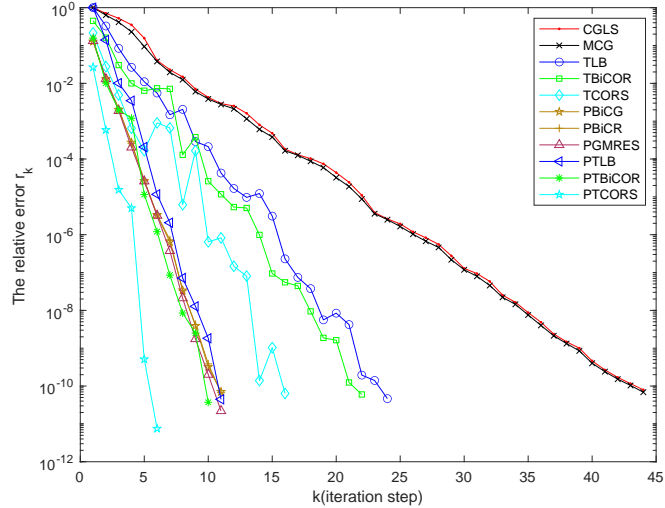}
    \caption{Plot of the relative error $r_{k}$ for Example \ref{ex4.3}.}\label{figc}
  \end{center}
\end{figure}

The initial solution $\mathcal{X}_{0}$ is selected as zero tensor. Algorithms \ref{alg3.2}-\ref{alg7} are used to solve \eqref{eq1.1} with $\textbf{A}_{i}$ given in \eqref{eq5.6}. These methods are compared with CGLS \cite{HM20}, MCG \cite{LM20}, PGMRES in \cite{CL12}, PBiCG and PBiCR in \cite{ZW21}.
Figure \ref{figc} shows that Algorithm \ref{alg7} converges fastest among all methods and Algorithm \ref{alg7} converges the second fastest among all methods, which are very similar to those in Figures \ref{figa} and \ref{figb}.

\section{Conclusion}\label{sec7}
This paper first presents a tensor Lanczos $\mathcal{L}$-Biorthogonalization (TLB) algorithm for solving the Sylvester tensor equation \eqref{eq1.1} based on the Lanczos $\mathcal{L}$-Biorthogonalization procedure. Then two improved methods based on the TLB algorithm are developed. The one is the biconjugate $\mathcal{L}$-orthogonal residual algorithm in tensor form (TBiCOR). The other is the conjugate $\mathcal{L}$-orthogonal residual squared algorithm in tensor form (TCORS). The preconditioner based on the nearest Kronecker product (NKP) are used to accelerate the TBiCOR and TCORS algorithms, thus we present preconditioned a preconditioned TBiCOR method and a preconditioned TCORS method. The convergence of these proposed algorithms are proved. Numerical examples show the advantage of  the  preconditioned TBiCOR and TCORS methods.

\section {Acknowledgments}\label{sec8}
The authors would like to thank the referees for their helpful comments which form the present version of this paper. The preconditioned methods are added according to one comment. Research by G.H. was supported in part by Application Fundamentals Foundation of STD of Sichuan (2020YJ0366) and Key Laboratory of bridge nondestructive testing and engineering calculation Open fund projects (2020QZJ03), and research by F.Y. was partially supported by NNSF (11501392) and SUSE (2019RC09).

\end{document}